\documentclass[11pt]{article}
\usepackage{amsmath, amsfonts, amssymb}
\usepackage{graphicx, amsthm}
\usepackage{cite}
\usepackage{color}
\usepackage{enumerate}
\usepackage{amsmath}
\usepackage{amsthm}
\usepackage{amssymb}
\usepackage{mathtools}
\usepackage{url}
\providecommand{\U}[1]{\protect\rule{.1in}{.1in}}

\newtheorem{theorem}{Theorem}[section]

\newtheorem{corollary}[theorem]{Corollary}
\newtheorem{proposition}[theorem]{Proposition}
\theoremstyle{definition}
\newtheorem{definition}[theorem]{Definition}

\theoremstyle{remark}
\newtheorem{remark}[theorem]{Remark}
\numberwithin{equation}{section}

\usepackage[margin=1.2in, top=1.2in, bottom=1.2in]{geometry}
\usepackage[vlined,ruled]{algorithm2e}
\SetKwComment{Comment}{[}{]}

\def\R{{\bf R}}
\def\X{\mathcal{X}}

\def\dist{{\rm dist}}

\def\diam{{\rm diam}}

\def\Lip{\textrm{Lip}}

\newcommand{\ls}{\operatornamewithlimits{limsup}}
\newcommand{\li}{\operatornamewithlimits{liminf}}
\newcommand{\argmin}{\operatornamewithlimits{argmin}}

\title{Nonsmooth optimization using Taylor-like models:  error bounds, convergence, and termination criteria}
\author{D. Drusvyatskiy\thanks{Department of Mathematics, University of Washington, 
		Seattle, WA 98195; \texttt{www.math.washington.edu/{\raise.17ex\hbox{$\scriptstyle\sim$}}ddrusv}. Research of Drusvyatskiy was partially supported by the AFOSR YIP award FA9550-15-1-0237 and by the US-Israel Binational Science Foundation Grant
		2014241.}\and 
	A.D. Ioffe\thanks{Department of Mathematics, Technion-Israel Institute of Technology, 32000 Haifa, Israel; \texttt{ioffe@tx.technion.ac.il}. Research supported in part by the US-Israel Binational Science Foundation Grant
	2014241.}
	\and A.S. Lewis\thanks{School of Operations Research and Information Engineering, Cornell University, Ithaca, New
		York, USA; \texttt{people.orie.cornell.edu/{\raise.17ex\hbox{$\scriptstyle\sim$}}aslewis}. Research supported in part by National
		Science Foundation Grant DMS-1208338 and by the US-Israel Binational Science Foundation Grant
		2014241.}}

\begin{document}
\date{}
\maketitle

\begin{abstract}
We consider optimization algorithms that successively minimize simple Taylor-like models of the objective function. Methods of Gauss-Newton type for minimizing the composition of a convex function and a smooth map are common examples. Our main result is an explicit relationship between the step-size of any such algorithm and the slope of the function at a nearby point. 
Consequently, we (1) show that the step-sizes can be reliably used to terminate the algorithm, (2) prove that as long as the 
 step-sizes tend to zero, every limit point of the iterates is stationary, and (3) show that conditions, akin to classical quadratic growth, imply that the step-sizes linearly bound the distance of the iterates to the solution set. The latter so-called error bound property is typically used to establish linear (or faster) convergence guarantees. Analogous results hold when the step-size is replaced by the square root of the decrease in the model's value. We complete the paper with extensions to when the models are minimized only inexactly.
\end{abstract}



\section{Introduction}
A basic algorithmic strategy for minimizing a function $f$ on $\R^n$ is to successively minimize simple ``models'' of the function, agreeing with $f$ at least up to first-order near the current iterate. We will broadly refer to such models as ``Taylor-like''. Some classical examples will help ground  the exposition. When $f$ is smooth, common algorithms given a current iterate $x_k$  declare the next iterate $x_{k+1}$ to be a minimizer of the quadratic model
$$m_k(x):=f(x_k)+\langle \nabla f(x_k),x-x_k\rangle+\frac{1}{2}\langle B_k(x-x_k),x-x_k\rangle.$$
When the matrix $B_k$ is a multiple of the identity, the scheme reduces to  gradient descent; when $B_k$ is the Hessian $\nabla^2 f(x_k)$,  one recovers Newton's method; adaptively changing $B_k$ based on accumulated  information covers Quasi-Newton algorithms. Higher-order models can also appear; the cubicly regularized Newton's method of Nesterov-Polyak \cite{cubic_nest} uses the models
$$m_k(x):=f(x_k)+\langle \nabla f(x_k),x-x_k\rangle+\frac{1}{2}\langle \nabla^2 f(x_k)(x-x_k),x-x_k\rangle +\frac{M}{6}\|x-x_k\|^3.$$
For more details on Taylor-like models in smooth minimization, see  Nocedal-Wright \cite{nocedal_wright}. 
   
The algorithmic strategy generalizes far beyond smooth minimization. One important arena, and the motivation for the current work, is the class of convex composite problems
\begin{equation}\label{eqn:convex_com_prob}
\min_x~ g(x) + h(c(x)).
\end{equation}
Here $g$ is a closed convex function (possibly taking infinite values), $h$ is a finite-valued Lipschitz  convex function, and $c$ is a smooth map. 
 Algorithms for this problem class have been studied extensively, notably in \cite{powell_paper,burke_com,yu_super,steph_conv_comp,fletcher_back,pow_glob} and more recently in 
  \cite{prx_lin,error_bound_d_lewis,composite_cart}.  Given a current iterate $x_k$, common algorithms declare the next iterate $x_{k+1}$ to be a minimizer of 
\begin{equation}\label{eqn:mod_con_comp}
m_k(x):=g(x)+h\Big(c(x_k)+\nabla c(x_k)(x-x_k)\Big)+\frac{1}{2}\langle B_k(x-x_k),x-x_k\rangle.
\end{equation}
The underlying assumption is that the minimizer of $m_k$ can be efficiently computed. This is the case for example, when interior-point methods can be directly applied to the convex subproblem or when evaluating $c$ and $\nabla c$ is already the computational bottleneck. The latter setting is ubiquitous in derivative free optimization; see for example the discussion in Wild  \cite{pounders}. 
The model $m_k$ in \eqref{eqn:mod_con_comp} is indeed Taylor-like, even when $g$ and $h$ are nonconvex, since the inequality $|m_k(y)-f(y)|\leq \frac{\Lip(h)\Lip(\nabla c)+\|B_k\|}{2}\|y-x_k\|^2$ holds for all points $y$, as the reader can readily verify.
When $B_k$ is a multiple of the identity, the resulting method is called the ``prox-linear algorithm'' in \cite{error_bound_d_lewis,prx_lin}, and it subsumes a great variety of schemes. 

In the setting $h=0$, the prox-linear algorithm reduces to the proximal-point method on the function $g$ \cite{prox_rock,martinet1,martinet2}. When $c$ maps to the real line and $h$ is the identity function, the scheme is the proximal gradient algorithm  on the function $g+c$ \cite{nest1,beck}. The Levenberg-Marquardt algorithm \cite{marq}  for nonlinear least squares  -- a close variant of Gauss-Newton -- corresponds to setting $g=0$ and $h=\|\cdot\|^2$. Allowing $B_k$ to vary with accumulated information yields variable metric variants of the aforementioned algorithms; see e.g. \cite{Scheinberg2016,Byrd:2016,burke_com}. Extensions where $h$ and $g$ are not necessarily convex, but are nonetheless simple, are also important and interesting, in large part because of nonconvex penalties and regularizers common  in machine learning applications.  Other important variants interlace the model minimization step with inertial corrector steps, such as in accelerated gradient methods \cite{nest_orig,ghadimi_lan }, cubically regularized Newton \cite{accel_newton}, and convex composite algorithms \cite{accel_prox_comp}.

 In this work, we take a broader view of nonsmooth optimization algorithms that use Taylor-like models.    
Setting the stage, consider the minimization problem 
$$\min_x~ f(x)$$
for an arbitrary lower-semicontinuous function $f$ on $\R^n$. The model-based algorithms we consider simply iterate the steps: $x_{k+1}$ is a minimizer of some model $f_{x_k}(\cdot)$ based at $x_k$. In light of the discussion above, we assume that the models $f_{x_k}$ approximate $f$ (uniformly) up to first-order, meaning$$|f_{x_k}(x)-f(x)|\leq \omega(\|x-x_k\|)\qquad \textrm{ for all }k\in \mathbb{N} \textrm{ and }x\in\R^n,$$
where $\omega$ is any $C^1$-smooth function satisfying $\omega(0)=\omega'(0)=0$. For uses of a wider class of models for bundle methods, based on cutting planes, see Noll-Prot-Rondepierre\cite{noll_model}.

In this great generality, we begin with the following basic question.
\begin{center}
When should one terminate an algorithm that uses  Taylor-like models? 
\end{center}
For smooth nonconvex optimization, the traditional way to reliably terminate the algorithm is to stop when the norm of the gradient at the current iterate is smaller than some tolerance. For nonsmooth problems, termination criteria based on optimality conditions along the iterates may be meaningless as they may never be satisfied even in the limit. For example, one can easily exhibit a convex composite problem so that the iterates generated by the prox-linear algorithm described above converge to a stationary point, while the optimality conditions at the iterates are not satisfied even in the limit.\footnote{One such univariate example is $\min_x f(x)=|\frac{1}{2}x^2+x|$.  
The prox-linear algorithm for convex composite minimization  \cite[Algorithm 5.1]{error_bound_d_lewis}  initiated to the right of  the origin -- a minimizer of $f$ -- will generate a sequence $x_k\to 0$ with $|f'(x_k)|\to 1$.}  Such lack of natural stopping criteria for nonsmooth first-order methods has been often remarked (and is one advantage of bundle-type methods).

There are, on the other hand, two appealing stopping criteria one can try: terminate the algorithm when either the step-size $\|x_{k+1}-x_{k}\|$ or the model improvement $f(x_{k})-\inf f_{x_k}$ is sufficiently small. We will prove that both of these simple termination criteria are indeed reliable in the following sense.  Theorem~\ref{thm:main} and Corollary~\ref{cor:inex_eval_base} show that if either the step-size $\|x_{k+1}-x_{k}\|$ or the model improvement $f(x_{k})-\inf f_{x_k}$ is small, then there exists a point $\hat{x}$ close to $x$ in both distance and in function value, which is nearly stationary for the problem. Determining the point $\hat{x}$ is not important; the only role of $\hat x$ is to certify that the current iterate $x_k$ is ``close to near-stationarity'' in the sense above. 
 Theorem~\ref{thm:main} follows quickly from Ekeland's variational principle \cite{var_princ} -- a standard variational analytic tool. For other uses of the technique in variational analysis, see for example the survey~\cite{ioffe_survey}.  Stopping criterion based on small near-by subgradients has appeared in many other contexts such as in descent methods of \cite{gold_eps_stat} and gradient sampling schemes of \cite{rob_grad_samp}.

Two interesting consequences for convergence analysis flow from there.
Suppose that the models are chosen in such a way that the step-sizes $\|x_{k+1}-x_k\|$ tend to zero. This assumption is often enforced by ensuring that  $f(x_{k+1})$ is smaller than $f(x_{k})$ by at least  a multiple of  $\|x_{k+1}-x_k\|^2$ (a sufficient decrease condition) using a line-search procedure or by safeguarding the minimal eigenvalue of $B_k$.
 Then assuming for simplicity that $f$ is continuous on its domain,  any limit point $x^*$ of the iterate sequence $x_k$ will be stationary for the problem (Corollary~\ref{cor:stat}).\footnote{By stationary, we mean that zero lies in the limiting subdifferential of the function at the point.} Analogous results hold with the step-size replaced by $f(x_{k})-\inf f_{x_k}$.

The subsequence convergence result is satisfying, since very little is assumed about the underlying algorithm. A finer analysis of linear, or faster, convergence rates relies on some regularity of the function $f$ near a limit point $x^*$ of the iterate sequence $x_k$. One of the weakest such regularity assumptions is that for all $x$ near $x^*$, the ``slope'' of $f$ at $x$  linearly bounds the distance of $x$  to the set of stationary points $S$ -- the ``error''. Here, we call this property the {\em slope error-bound}.
To put it in perspective, we note that the slope error-bound
always entails a classical quadratic growth condition away from $S$ (see \cite{crit_semi,tilt_other}), and is equivalent to it whenever $f$ is convex (see \cite{fran_sub,klat_kum_book}). Moreover, as an aside, we observe in Theorem~\ref{thm:slope_kl_ineq} and Proposition~\ref{prop:conv_kurd_loj} that under mild conditions, the slope error-bound is equivalent to the  ``Kurdyka-{\L}ojasiewicz inequality'' with exponent $1/2$ -- an influential condition also often used to prove linear convergence.

Assuming the slope error-bound, a typical convergence analysis strategy aims to deduce that the step-sizes
$\|x_{k+1}-x_k\|$ linearly bound the distance $\dist(x_k;S)$. Following Luo-Tseng \cite{error_conv}, we call the latter property the {\em step-size error-bound}. We show in Theorem~\ref{thm:slop_subreg} that the slope error-bound indeed always implies the step-size error-bound, under the common assumption that the growth function $\omega(\cdot)$ is a quadratic. The proof is a straightforward consequence of the relationship we have established between the step-size and the slope at a nearby point -- underscoring the power of the technique.

In practice, exact minimization of the model function $f_{x_k}$ can be impossible. Instead, one can obtain a point $x_{k+1}$ that is only nearly optimal  or nearly stationary for the problem $\min f_{x_k}$. Section~\ref{sec:inexact} shows that all the results above generalize to this more realistic setting. In particular, somewhat surprisingly, we argue that limit points of the iterates will be stationary even if the tolerances
on optimality (or stationarity) and the step-sizes $\|x_{k+1}-x_k\|$ tend to zero at independent rates. The arguments in this inexact setting follow by applying the key result, Theorem~\ref{thm:main}, to small perturbations of $f$ and $f_{x_k}$, thus illustrating the flexibility of the theorem.

The convex composite problem \eqref{eqn:convex_com_prob} and the prox-linear algorithm (along with its variable metric variants) is a fertile application arena for the techniques developed here. An early variant of the key Theorem~\ref{thm:main} in this setting appeared recently in \cite[Theorem 5.3]{error_bound_d_lewis} and was used there to establish sublinear and linear convergence guarantees for the prox-linear method. We review these results and derive extensions in Section~\ref{sec:conc_comp_ill}, as an illustration of our techniques. An important deviation of ours from earlier work is the use of the step-size as the fundamental analytic tool, in contrast to the $\Delta$ measures of Burke \cite{burke_com} and the criticality measures in Cartis-Gould-Toint \cite{composite_cart}. To the best of our knowledge, the derived relationship between the step-size and stationarity at a nearby point is entirely new. The fact that the slope error-bound implies that both the step-size and the square root of the model improvement linearly bounds the distance to the solution set (step-size and model error-bounds) is entirely new as well; previous related results have relied on polyhedrality assumptions on $h$.

Though the discussion above takes place over the Euclidean space $\R^n$, the most appropriate setting for most of our development is over an arbitrary complete metric space. This is the setting of the paper.
The outline is as follows. In Section~\ref{sec:notation}, we establish basic notation and recall Ekeland's variational principle. Section~\ref{sec:main} contains our main results. Section~\ref{sec:conc_comp_ill} instantiates the techniques for the prox-linear algorithm in composite minimization, while Section~\ref{sec:inexact} explores extensions when the subproblems are solved inexactly. 

\section{Notation}\label{sec:notation}
Fix a complete metric space $\mathcal{X}$ with the metric $d(\cdot,\cdot)$. We denote the open unit ball of radius $r>0$ around a point $x$ by ${\bf B}_r(x)$.
The distance from  $x$ to a set $Q\subset \mathcal{X}$ is 
$$\dist(x;Q):=\inf_{y\in Q}~ d(x,y).$$
We will be interested in minimizing functions mapping $\mathcal{X}$ to the extended real line $\overline{\R}:=\R\cup\{\pm\infty\}$. A function $f\colon\X\to\overline \R$ is called  {\em lower-semicontinuous} (or {\em closed}) if the inequality $\li_{x\to \bar x} f(x)\geq f(\bar x)$ holds for all points $\bar x\in \X$. 

Consider a closed function $f\colon\mathcal{X}\to\overline{\R}$ and a point $\bar{x}$ with $f(\bar x)$ finite.
The {\em slope} of $f$ at  $\bar{x}$ is simply its maximal instantaneous rate of decrease:
$$|\nabla f|(\bar x):=\ls_{x\to \bar x}~\frac{(f(\bar x)-f(x))^+}{d(\bar{x},x)}.$$
Here, we use the notation $r^+=\max\{0,r\}$. 
If $f$ is a differentiable function on a Euclidean space, the slope $|\nabla f|(\bar x)$ simply coincides with the norm of the gradient $\|\nabla f(\bar x)\|$, and hence the notation. For a convex function $f$, the slope $|\nabla f|(\bar x)$ equals the norm of the shortest subgradient $v\in\partial f(\bar x)$. For more details on the slope and its uses in optimization, see the survey \cite{ioffe_survey} or the thesis \cite{thesis}.

The function $x\mapsto |\nabla f|(x)$ lacks basic lower-semicontinuity  properties. As a result when speaking of algorithms, it is important to introduce the {\em limiting slope}
$$\overline{|\nabla f|}(\bar x):=\li_{x\to \bar{x},\, f(x)\to f(\bar x)}~ |\nabla f|(x).$$
In particular, if $f$ is continuous on its domain, then $\overline{|\nabla f|}$ is simply the lower-semicontinuous envelope of $|\nabla f|$. We say that a point $\bar x$ is {\em stationary} for $f$ if equality $\overline{|\nabla f|}(\bar x)=0$ holds.

We will be interested in locally approximating functions up to first-order. Seeking to measure the ``error in approximation'', we introduce the following definition. 
 


\begin{definition}[Growth function]
A differentiable univariate function $\omega\colon\R_+\to\R_+$ is called a {\em growth function} if it satisfies $\omega(0)=\omega'(0)=0$ and $\omega'>0$ on $(0,\infty)$. If in addition, equalities $\lim_{t\to 0} \omega'(t)=\lim_{t\to 0} \omega(t)/\omega'(t)=0$ hold, we say that $\omega$ is a {\em proper growth function}.
\end{definition}

\noindent The main examples of proper growth functions are $\omega(t):=\frac{\eta}{r}\cdot t^{r}$ for real $\eta>0$ and $r >1$. 

The following result, proved in \cite{var_princ}, will be our main tool. The gist of the theorem is that if a point $\bar x$ nearly minimizes a closed function, then $\bar x$ is close to a a true minimizer of a slightly perturbed function. 
\begin{theorem}[Ekeland's variational principle] \label{thm:eke}
	Consider a closed function $g\colon\X\to\overline{\R}$ that is bounded from below. Suppose that for some $\epsilon>0$ and $\bar x\in \R^n$, we have $g(\bar x)\leq \inf g+\epsilon$. Then for any real $\rho >0$, there exists a point $\hat x$ satisfying
	\begin{enumerate}
		\item $g(\hat x)\leq g(\bar x)$,
		\item $d(\bar x,\hat x)\leq \epsilon/\rho$,
		\item\label{it:minim} $\hat x$ is the unique minimizer of the perturbed function $x\mapsto g(x)+\rho\cdot  d(x,\hat x)$. 
	\end{enumerate}	
\end{theorem}
Notice that property \ref{it:minim} in Ekeland's principle directly implies the inequality $|\nabla g|(\hat x)\leq \rho$. Thus if a point $\bar x$ nearly minimizes $g$, then the slope of $g$ is small at some nearby point.

\subsection{Slope and subdifferentials}
The slope is a purely metric creature. However, for a function $f$ on $\R^n$, the slope  is closely related to ``subdifferentials'', which may be more familiar to the audience. We explain the relationship here following \cite{ioffe_survey}. Since the discussion will not be used in the sequel, the reader can safely skip it and move on to Section~\ref{sec:main}.

A vector $\bar v\in \R^n$ is called a {\em Fr\'{e}chet subgradient} of a function $f\colon\R^n\to\overline\R$ at a point $\bar x$ if the inequality 
$$f(x)\geq f(\bar x)+\langle\bar v,x-\bar x \rangle +o(\|x-\bar x\|)\qquad \textrm{ holds as } x\to\bar x.$$
The set of all Fr\'{e}chet subgradients of $f$ at $\bar x$ is the {\em Fr\'{e}chet subdifferential} and is denoted by $\hat \partial f(\bar x)$. The connection of the slope $|\nabla f|(\bar x)$ to subgradients is immediate. A vector $\bar v$ lies in $\hat \partial f(\bar x)$ if and only if the slope of the linearly tilted function $f(\cdot)-\langle \bar v,\cdot\rangle$ at $\bar x$ is zero. Moreover the inequality
\begin{equation}\label{eqn:ineq_slope_subreg}
|\nabla f|(\bar x)\leq \dist (0,\hat \partial f(\bar x)) \qquad \textrm{  holds}.
\end{equation}
The {\em limiting subdifferential} of $f$ at $\bar x$, denoted $\partial f(\bar x)$, consists of all vectors $\bar v$ such that there exists sequences $x_i$ and $v_i\in\hat\partial f(x_i)$ satisfying
 $(x_,f(x_i),v_i)\to(\bar x, f(\bar x),\bar v)$.
Assuming that $f$ is closed,  a vector $\bar v$ lies in $ \partial f(x)$ if and only if the limiting slope of the linearly tilted function $f(\cdot)-\langle \bar v,\cdot\rangle$ at $\bar x$ is zero. Moreover, Proposition 4.6 in \cite{curves} shows that the exact equality
\begin{equation}\label{eqn:equal_slop_sub}
\overline{|\nabla f|}(\bar x)= \dist (0, \partial f(\bar x)) \qquad \textrm{  holds}.
\end{equation}
In particular, stationarity of $f$ at $\bar x$ amounts to the inclusion $0\in \partial f(\bar x)$.

\section{Main results}\label{sec:main}
For the rest of the paper, fix a closed function $f\colon\X\to\overline{\R}$ on a complete metric space $\X$, and a point $x$ with $f(x)$ finite.
The following theorem is our main result. It shows that for any function $f_x(\cdot)$ (the ``model''),
such that the error in approximation $|f_x(y)-f(y)|$ is controlled by a growth function of 
the norm $d(x,y)$,  the distance between $x$ and the minimizer $x^+$ of $f_x(\cdot)$ prescribes near-stationarity of $f$ at some nearby point $\hat x$. 


\begin{theorem}[Perturbation result]\label{thm:main}
Consider a closed function $f_x\colon\X\to\overline\R$ such  that the inequality
$$|f_x(y)-f(y)|\leq \omega(d(x,y)) \quad \textrm{ holds for all }  y\in\mathcal X,$$
where $\omega$ is some growth function, 
and let $x^+$ be a minimizer of $f_x$. 
If $x^+$ coincides with $x$, then the slope $|\nabla f|(x)$ is zero. 
On the other hand, if $x$ and $x^+$ are distinct, then 
 there exists a point $\hat x\in \X$ satisfying
\begin{enumerate}
\item {\bf (point proximity)} $\quad d(x^+,\hat{x})\leq 2\cdot \frac{\omega(d(x^+,x))}{\omega'(d(x^+,x))}$,
\item {\bf (value proximity)} $\quad f(\hat x)\leq f(x^+)+\omega(d(x^+,x))$,
\item {\bf (near-stationarity)} $\quad |\nabla f|(\hat x)\leq \omega'(d(x^+,x))+\omega'(d(\hat x,x))$.
\end{enumerate}
\end{theorem}
\begin{proof}
A quick computation shows the equality $|\nabla f_x|(x)=|\nabla f|(x)$.	Thus if $x^+$ coincides with $x$, the slope $|\nabla f|(x)$ must be zero, as claimed. Therefore, for the remainder of the proof, we will assume that $x^+$ and $x$ are distinct.
 	
Observe now the inequality
$$f(y)\geq f_x(y)-\omega(d(x,y))\geq f_x(x^+)-\omega(d(x,y)).$$ Define  the function $g(y):=f(y)+\omega(d(x,y))$ and note $\inf g \geq f_x(x^+)$. We deduce  
\begin{equation}\label{eqn:eps_close}
g(x^+)-\inf g=f(x^+)-f_x(x^+)+\omega(d(x^+,x))\leq 2\cdot\omega(d(x^+,x)).
\end{equation}
An easy argument now shows the inequality
$$|\nabla g|(z)\geq |\nabla f|(z)-\omega'(d(z,x))\qquad \textrm{for all } z\in \X.$$
Setting  $\epsilon:=2\omega(d(x^+,x))$ and applying Ekeland's variational principle (Theorem~\ref{thm:eke}), we obtain for any $\rho>0$ a point $\hat{x}$ satisfying $$g(\hat x)\leq g(x^+),\qquad d(x^+,\hat x)\leq \frac{\epsilon}{\rho}, \qquad \textrm{and }\qquad |\nabla g|(\hat x)\leq \rho.$$
 We conclude $|\nabla f|(\hat x)\leq \rho+\omega'(d(\hat x, x)).$
Setting $\rho:=\omega'(d(x^+,x))$ yields the result.
\end{proof}

Note that the distance $d(\hat x,x)$ appears on the right hand-side of the near-stationarity property. By the triangle-inequality and point proximity, however, it can be upper bounded by $d(x^+,x)+2\cdot \frac{\omega(d(x^+,x))}{\omega'(d(x^+,x))}$, a quantity independent of $\hat x$. 

To better internalize this result, let us look at the most important setting of Theorem~\ref{thm:main} where the growth function is a quadratic
$\omega(t)=\frac{\eta}{2}t^{2}$ for some real $\eta>0$. 
\begin{corollary}[Quadratic error]\label{cor:quad_gr}
Consider a closed function $f_x\colon\X\to\overline\R$ and suppose that with some real $\eta>0$ the inequality
$$|f_x(y)-f(y)|\leq \frac{\eta}{2} \cdot d^{2}(x,y) \quad \textrm{holds for all }  y\in\mathcal X.$$
Define $x^+$ to be the minimizer of $f_x$. Then there exists a point $\hat x\in \R^n$ satisfying
\begin{enumerate}
\item {\bf (point proximity)} $\quad d(x^+,\hat{x})\leq  d(x^+,x)$,
\item {\bf (value proximity)} $\quad f(\hat x)\leq f(x^+)+\frac{\eta}{2}\cdot d^{2}(x^+,x)$,
\item {\bf (near-stationarity)} $\quad |\nabla f|(\hat x)\leq 5\eta \cdot d(x^+,x)$.
\end{enumerate}
\end{corollary}

An immediate consequence of Theorem~\ref{thm:main} is the following subsequence convergence result.

\begin{corollary}[Subsequence convergence to stationary points]\label{cor:stat}{\hfill \\ }
Consider a sequence of points $x_k$ and closed functions $f_{x_k}\colon\mathcal X\to\overline{\R}$ satisfying $x_{k+1}=\argmin_y f_{x_k}(y)$ and $d(x_{k+1},x_{k})\to 0$. Suppose moreover that the inequality 
$$|f_{x_k}(y)-f(y)|\leq \omega(d(y,x_k))\qquad \textrm{ holds for all indices } k \textrm{ and points } y\in \mathcal X,$$ 
where $\omega$ is a proper growth function. 
If $(x^*,f(x^*))$ is a limit point of the sequence $(x_{k},f(x_k))$, then  $x^*$ is 
stationary for $f$.
\end{corollary}
\begin{proof}
Fix a subsequence $x_{k_i}$ with $(x_{k_i},f(x_{k_i}))\to (x^*,f(x^*))$, and
consider the points $\hat x_{k_i}$ guaranteed to exist by Theorem~\ref{thm:main}. 
By  point proximity,  $d(x_{k_i},\hat x_{k_i-1})\leq \frac{\omega(d(x_{k_i},x_{k_i-1}))}{\omega'(d(x_{k_i},x_{k_i-1}))}$, and the fact that the right hand-side tends to zero, we conclude that $\hat x_{k_{i}-1}$ converge to $x^*$. The functional proximity,
$f(\hat x_{k_i-1})\leq f(x_{k_i})+\omega(d(x_{k_i},x_{k_i-1}))$ implies $\ls_{i\to\infty} f(\hat x_{k_i-1})\leq \ls_{i\to\infty} f(x_{k_i})=f(x^*)$. Lower-semicontinuity of $f$ then implies $\lim_{i\to\infty} f(\hat x_{k_i-1})=f(x^*)$. Finally, the near-stationarity, $$|\nabla f|(\hat x_{k_i-1})\leq \omega'(d(x_{k_i}, x_{k_i-1}))+\omega'(d(\hat x_{k_i-1},x_{k_i-1})),$$ implies $|\nabla f|(\hat x_{k_i-1})\to 0$. Thus $x^*$ is a stationary point of $f$.
\end{proof}

\begin{remark}[Asymptotic convergence to critical points]
Corollary~\ref{cor:stat} proves something stronger than stated. An unbounded sequence $z_k$ is {\em asymptotically critical} for $f$ it satisfies $ |\nabla f|(z_k)\to 0$. The proof of Corollary~\ref{cor:stat} shows that
if the sequence $x_k$ is unbounded, then there exists an asymptotically critical sequence $z_k$ satisfying $d(x_k,z_k)\to 0$.
\end{remark}

Corollary~\ref{cor:stat} is fairly satisfying since very little is assumed about the model functions. More sophisticated linear, or faster, rates of convergence rely on some regularity of the function $f$ near a limit point $x^*$ of the iterate sequence $x_k$. Let $S$ denote the set of stationary points of $f$. One of the weakest such regularity assumptions is that the slope $|\nabla f|(x)$ linearly bounds the distance $\dist(x;S)$ for all $x$ near $x^*$. Indeed, this property, which we call the {\em slope error-bound}, always entails a classical quadratic growth condition away from $S$ (see \cite{crit_semi,tilt_other}), and is equivalent to it whenever $f$ is a convex function on $\R^n$ (see \cite{fran_sub,klat_kum_book}).

Assuming such regularity, a typical convergence analysis strategy thoroughly explored by  Luo-Tseng \cite{error_conv}, aims to deduce that the step-sizes
$d(x_{k+1},x_k)$ linearly bound the distance $\dist(x_k;S)$. The latter is called the {\em step-size error-bound property}. We now show that slope error-bound always implies the step-size error-bound, under the mild and natural assumption that the models $f_{x_k}$ deviate form $f$ by a quadratic error in the distance.

\begin{theorem}[Slope and step-size error-bounds]\label{thm:slop_subreg}
Let $S$ be an arbitrary set and fix a point $x^*\in S$ satisfying the condition
 \begin{itemize}
\item {\bf(Slope error-bound)}$\quad \dist(x;S)\leq L\cdot |\nabla f|(x)\quad \text{ for all }\quad x\in {\bf B}_{\gamma}(x^*)$.
\end{itemize}
Consider a closed function $f_{x}\colon\mathcal X\to\overline{\R}$ and suppose that for some $\eta>0$ the  inequality 
$$|f_{x}(y)-f(y)|\leq \frac{\eta}{2}d^2(y,x)\qquad \textrm{ holds for all }  y\in \mathcal X.$$ 
Then letting $x^+$ be any minimizer of $f_x$, the following holds:
\begin{itemize}
\item {\bf(Step-size error-bound)}$\quad \dist(x,S)\leq (3L\eta+2)\cdot d(x^+,x)\quad \text{ when }\quad x,x^+\in {\bf B}_{\gamma/3}(x^*).$
\end{itemize}
\end{theorem}
\begin{proof}
Suppose that the points $x$ and $x^+$ lie in ${\bf B}_{\gamma/3}(x^*)$. Let $\hat x$ be the point guaranteed to exist by Corollary~\ref{cor:quad_gr}. We deduce 
\begin{align*}
d(\hat x,x^*)&\leq d(\hat x,x^+)+d(x^+,x^*)\leq d(x^+,x)+d(x^+,x^*)<\gamma.
\end{align*}
	Thus $\hat x$ lies in ${\bf B}_{\gamma}(x^*)$ and we obtain 
	\begin{align*}
	L\cdot |\nabla f|(\hat x)\geq \dist\left(\hat x; S\right)&\geq \dist\left(x; S\right)-d(x^+,\hat x)-d(x^+,x)\\
	&\geq  \dist\left(x; S\right)-2d(x^+,x).
	\end{align*}
	Taking into account the inequality $|\nabla f|(\hat x)\leq 3\eta\cdot d(x^+,x)$, we conclude
$$\dist\left(x; S\right)\leq (3L\eta+2)\cdot d(x^+,x),$$ as claimed.
\end{proof}

\begin{remark}[Slope and subdifferential error-bounds]
It is instructive to put the slope error-bound property in perspective for those more familiar with subdifferentials. To this end, suppose that $f$ is defined on $\R^n$ and consider the {\em subdifferential error-bound} condition
\begin{equation}\label{eqn:subdiff_subreg}
\dist(x;S)\leq L\cdot \dist(0;\hat\partial f(x))\quad \text{ for all }\quad x\in {\bf B}_{\gamma}(x^*). 
\end{equation}
Clearly in light of the inequality \eqref{eqn:ineq_slope_subreg}, the slope error-bound implies the subdifferential error-bound \eqref{eqn:subdiff_subreg}. Indeed, the slope  and subdifferential error-bounds are equivalent. To see this, suppose \eqref{eqn:subdiff_subreg} holds and consider an arbitrary point $x\in {\bf B}_{\gamma}(x^*)$. Appealing to the equality \eqref{eqn:equal_slop_sub}, we obtain sequences $x_i$ and $v_i\in \hat\partial f(x_i)$  satisfying $x_i\to x$ and $\|v_i\|\to \overline{|\nabla f|}(x)$. Inequality
\eqref{eqn:subdiff_subreg} then implies $\dist(x_i;S)\leq L\cdot \|v_i\|$ for each sufficiently large index $i$. Letting $i$ tend to infinity yields the inequality,
$\dist(x;S)\leq L\cdot \overline{|\nabla f|}(x)\leq L\cdot |\nabla f|(x)$, and therefore the slope error-bound is valid.
\end{remark}

\smallskip

Lately, a different condition called the {\em Kurdyka-{\L}osiewicz inequality} \cite{Lewis-Clarke,Kurd} with exponent 1/2 has been often used to study linear rates of convergence. The manuscripts \cite{desc_semi,prox_lin_nonconv} are influential examples. We finish the section with the observation that the  Kurdyka-{\L}osiewicz inequality always implies the slope error-bound relative to a sublevel set $S$; that is, the K{\L} inequality is no more general than the slope error-bound. A different argument for (semi) convex functions based on subgradient flow appears in \cite[Theorem 5]{bolte_complex_kur} and \cite{pl_ineq}.
In Proposition~\ref{prop:conv_kurd_loj} we will also observe that the converse implication holds for all prox-regular functions. Henceforth, we will use the sublevel set notation $[f\leq b]:=\{x:f(x)\leq b\}$ and similarly $[a<f< b]:=\{x: a< f(x)< b\}$.

\begin{theorem}[K{\L}-inequality implies the slope error-bound]\label{thm:slope_kl_ineq}
Suppose that there is a nonempty open set $\mathcal{U}$ in $\mathcal{X}$ such that the inequalities
$$(f(x)-f^*)^{\theta}\leq \alpha\cdot  |\nabla f(x)|\qquad \textrm{hold for all }\qquad x\in \mathcal U\cap [f^*<f<r],$$
where $\theta\in (0,1)$, $\alpha>0$,  $f^*$, and $r>f^*$ are real numbers.
Then there exists a nonempty  open set $\widehat{ \mathcal U}$ and a real number $\hat r$ so that the inequalities
$$d(x;[f\leq f^*])\leq \frac{\alpha^{\theta^{-1}}}{1-\theta} \cdot |\nabla f|^{\frac{1-\theta}{\theta}}(x)\qquad \textrm{ hold for all }\qquad x\in \widehat{\mathcal{U}}\cap [f^*<f<\hat r].$$
 In the case $\mathcal U=\mathcal X$, we can ensure $\widehat{\mathcal U}=\mathcal X$ and $\hat r=r$. 
\end{theorem}
\begin{proof}
	Define the function $g(x)=(\max\{0,f(x)-f^*\})^{1-\theta}$. Note the inequality
	$|\nabla g|(x)\geq \frac{1-\theta}{\alpha}$ for all  $x\in \mathcal U\cap [f^*<f<r]$.
Let $R>0$ be strictly smaller than the largest radius of a ball contained in $\mathcal{U}$ and define $\varepsilon:=\min\left\{r-f^*,\frac{(1-\theta)R}{\alpha}\right\}$. 	
Define the nonempty set $\widehat{\mathcal{U}}:=\{x\in \mathcal{U}: {\bf B}_{R}(x)\subseteq \mathcal{U}\}$ and fix a point $x\in \widehat{\mathcal{U}}\cap [f^*<f<f^*+\varepsilon]$. 

Observe now for any point $u\in [f^*<f<f^*+\varepsilon]$ with $d(x,u)\leq R$, the inclusion $u\in \mathcal{U}\cap [f^*<f<r]$ holds, and hence $|\nabla g|(u)\geq \frac{1-\theta}{\alpha}$. Appealing to \cite[Lemma 2.5]{curves} (or \cite[Chapter 1, Basic Lemma]{ioffe_survey}), we deduce the estimate 
$$d(x;[f\leq f^*])\leq \frac{\alpha}{1-\theta}\cdot  g(x)=\frac{\alpha}{1-\theta}\cdot (f(x)-f^*)^{1-\theta}\leq \frac{\alpha^{\theta^{-1}}}{1-\theta} \cdot(|\nabla f|(x))^{\frac{1-\theta}{\theta}}.$$ The proof is complete. 
\end{proof}

The converse of Theorem~\ref{thm:slope_kl_ineq} holds for ``prox-regular functions'' on $\R^n$, and in particular for ``lower-$C^2$ functions''. The latter are functions $f$ on $\R^n$ such that around each point there is a neighborhood $\mathcal U$ and a real $l>0$ such that $f+\frac{l}{2}\|\cdot\|^2$ is convex on $\mathcal U$ .

\begin{proposition}[Slope error-bound implies  K\L-inequality under prox-regularity]\label{prop:conv_kurd_loj} {\hfill \\ }
	Consider a closed function $f\colon\R^n\to\overline{\R}$. Fix a real number $f^*$ and a nonempty set $S\subseteq [f\leq f^*]$. Suppose that there is a set $\mathcal{U}$, and constants $L$, $l$, $\epsilon$, and $r>f^*$ such that
	the inequalities
	$$f(y)\geq f(x)+\langle v,y-x\rangle-\frac{l}{2}\|y-x\|^2,$$
	$$\dist(x;S)\leq L \cdot \dist(0;\partial f(x)),$$
	hold for all $x\in \mathcal{U}\cap [f^*<f<r]$, $y\in \mathcal{X}$, and $v\in \partial f(x)\cap{\bf B}_{\epsilon}(0)$. Then the inequalities
	$$\sqrt{f(x)-f^*}\leq \sqrt{L+lL^2/2}\cdot \dist(0;\partial f(x)),$$	
	hold for all $x\in \mathcal{U}\cap [f^*<f<\hat r]$ where we set $\hat r:=\min\{r,(L+lL^2/2)\epsilon^2\}$.
\end{proposition}
\begin{proof}
	Consider a point $x \in \mathcal{U}\cap [f^*<f<\hat r]$. Suppose first $\dist(0;\partial f(x))\geq\epsilon$. Then we deduce $\sqrt{f(x)-f^*}\leq \sqrt{\hat r}\leq  \sqrt{L+lL^2/2} \cdot \epsilon\leq \sqrt{L+lL^2/2}\cdot \dist(0;\partial f(x))$, as claimed. Hence we may suppose there exists a subgradient $v\in \partial f(x)\cap{\bf B}_{\epsilon}(0)$. We deduce
	$$f^*\geq f(y)\geq f(x)+\langle v,y-x\rangle-\frac{l}{2}\|y-x\|^2\geq f(x)-\|v\|\cdot\|y-x\|-\frac{l}{2}\|y-x\|^2.$$ 
	Choosing $v$, $y$ such that $\|v\|$ and $\|y-x\|$ attain $\dist(0;\partial f(x))$ and 
	$\dist(x;S)$, respectively, we deduce
	$f(x)-f^*\leq \left(L+\frac{lL^2}{2}\right)\cdot\dist^2(0;\partial f(x))$. The result follows. 
\end{proof}

\section{Illustration: convex composite minimization}\label{sec:conc_comp_ill}
In this section, we briefly illustrate the results of the previous section in the context of composite minimization, and recall some consequences already derived in \cite{error_bound_d_lewis} from preliminary versions of the material presented in the current paper. This section will not be used in the rest of the paper, and so the reader can safely skip it if needed. 

The notation and much of discussion follows that set out in \cite{error_bound_d_lewis}. Consider the minimization problem
$$\min_x~ f(x):=g(x)+h(c(x)),$$
where $g\colon\R^n\to\overline \R$ is a closed convex function, $h\colon\R^m\to\R$ is a finite-valued $l$-Lipschitz convex function, and $c\colon\R^n\to\R^m$ is a $C^1$-smooth map with the Jacobian $\nabla c(\cdot)$ that is $\beta$-Lipschitz continuous.  Define the model function 
$$f_x(y):=g(y)+h\Big(c(x)+\nabla c(x)(y-x)\Big)+\frac{l\beta}{2}\|y-x\|^2.$$
One can readily verify the inequality
$$0\leq f_x(y)-f(y)\leq \frac{l\beta}{2}\|y-x\|^2\qquad \textrm{ for all }x,y\in \R^n.$$
In particular, the models $f_x$ are ``Taylor-like''. The prox-linear algorithm iterates the steps
\begin{equation}\label{eqn:iterate}
x_{k+1}=\argmin_y~ f_{x_k}(y).
\end{equation}
The following is a rudimentary convergence guarantee of the scheme \cite[Section 5]{error_bound_d_lewis}:
\begin{equation}\label{eqn:rate_prox_lin}
\sum^k_{i=1} \|x_{i+1}-x_i\|^2\leq \frac{2(f(x_1)-f^*)}{l\beta},
\end{equation}
where $f^*$ is the limit of the decreasing sequence $\{f(x_k)\}$. In particular, the step-sizes $\|x_{i+1}-x_i\|$ tend to zero. Moreover, one can readily verify that for any limit point $x^*$ of the iterate sequence $x_k$, equality $f^*=f(x^*)$ holds. Consequently, by Corollary~\ref{cor:stat}, the point $x^*$ is stationary for $f$:
$$0\in \partial f(x^*)=\partial g(x^*)+\nabla c(x^*)^T\partial h(c(x^*)).$$
We note that stationarity of the limit point $x^*$ is well-known and can be proved by other means; see for example the discussion in \cite{composite_cart}.
From \eqref{eqn:rate_prox_lin}, one also concludes the rate
$$\min_{i=1,\ldots,k} \|x_{i+1}-x_i\|^2\leq \frac{2(f(x_1)-f^*)}{l\beta\cdot k}.$$
What is the relationship of this rate to near-stationary of the iterate $x_k$? Corollary~\ref{cor:quad_gr} shows that 
after $\frac{2l\beta(f(x_1)-f^*)}{25\cdot \epsilon^2}$ iterations, one is guaranteed to find an iterate $x_k$ such that there exists a point $\hat x$ satisfying 
$$\left.\begin{array}{c}
5l\beta\cdot\|\hat x-x_{k+1}\|\\
5\sqrt{2l\beta}\cdot\sqrt{(f(\hat x)-f(x_{k+1}))^+}\\
\dist(0;\partial f(\hat x))
\end{array}\right\}~\leq~ \epsilon.
$$

Let us now move on to linear rates of convergence. Fix a limit point $x^*$ of the iterate sequence $x_k$ and let $S$ be the set of stationary points of $f$. Then Theorem~\ref{thm:slop_subreg} shows that the regularity condition
\begin{itemize}
	\item {\bf(Slope error-bound)} $\quad \dist(x;S)\leq L\cdot \dist(0;\partial f(x))\quad \text{ for all }\quad x\in {\bf B}_{\gamma}(x^*)$.
\end{itemize}
implies 
\begin{itemize}
	\item {\bf(Step-size error-bound)}$\quad \dist(x_k,S)\leq (3l L\beta+2)\cdot \|x_{k+1}-x_k\|\quad \text{ when }\quad x_k,x_{k+1}\in {\bf B}_{\gamma/3}(x^*).$
\end{itemize}
Additionally, in the next section (Corollary~\ref{cor:func_error}) we will show that the slope error-bound also implies 
\begin{itemize}
	\item {\bf(Model error-bound)}$$\dist(x_k;S)\leq \left(L\sqrt{12l \beta}+\frac{2}{\sqrt{3\eta}}\right)\cdot \sqrt{f(x_k)-\inf f_{x_k}},$$ whenever $f(x_k)-\inf f_{x_k}< \frac{3l\beta\gamma^2}{16}$ and $x_k$ lies in ${\bf B}_{\gamma/2}(x^*)$.
\end{itemize}

It was, in fact, proved in \cite[Theorem 5.10]{error_bound_d_lewis} that the slope and step-size error bounds are equivalent up to a change of constants. Moreover, as advertised in the introduction, the above implications were used in \cite[Theorem 5.5]{error_bound_d_lewis} to show that if the slope error-bound holds then the function values converge linearly:
$$f(x_{k+1})-f^*\leq q (f(x_k)-f^*) \qquad \textrm{ for all large }k,$$ 
where 
$$q\approx 1-\frac{1}{(L\beta l)^2}.$$ 
The rate improves to $q\approx 1-\frac{1}{L\beta l}$ under a stronger regularity condition, called tilt-stability \cite[Theorem 6.3]{error_bound_d_lewis}; the argument of the better rate again crucially employs a comparison of step-lengths and subgradients at near-by points. 

Our underlying assumption is that the models $f_{x_{k}}$ are easy to minimize, by an interior point method for example. This assumption may not be realistic in some large-scale applications. Instead, one must solve the subproblems \eqref{eqn:iterate} inexactly by a first-order method. How should one choose the accuracy?

Fix a sequence of tolerances $\varepsilon_k>0$ and suppose that each iterate $x_{k+1}$ satisfies $$f_{x_k}(x_{k+1})\leq \varepsilon_k+\inf f_{x_k}.$$
Then one can establish the estimate
$$\min_{i=1,\ldots,k} \|\bar{x}_{i+1}-x_i\|^2\leq \frac{2(f(x_1)-f^*)+\sum_{i=1}^k \varepsilon_k}{l\beta\cdot k},$$
where $\bar{x}_{i+1}$ is the true minimizer of $f_{x_i}$. The details will appear in a forthcoming paper. Solving \eqref{eqn:iterate} to $\varepsilon_k$ accuracy can be done in multiple ways using a saddle point reformulation. We treat the simplest case here, where $h$ is a support function $h(y)=\sup_{z\in \mathcal{Z}}\langle z,y\rangle$ of a closed bounded set $\mathcal{Z}$ -- a common setting in applications. We can then write
$$f_x(y)= g(y)+\frac{l\beta}{2}\|y-x\|^2+\max_{z\in \mathcal{Z}}\,\left\{ \langle z,c(x)+\nabla c(x)(y-x)\rangle\right\}.$$
Dual to the problem $\min_y f_{x}(y)$ is the maximization problem 
$$\max_{z\in \mathcal{Z}} \phi(z):=\langle c(x)-\nabla c(x),z\rangle-\left(g+\frac{l\beta}{2}\|\cdot-x\|^2\right)^{\star}(-\nabla c(x)^Tz).$$
For such problems, there are primal-dual methods \cite[Method 2]{excess} that generate iterates $y_k$ and $z_k$ satisfying
$$f(y_k)-\phi(z_k)\leq \frac{4\left(\frac{\|\nabla c(x)\|^2}{l\beta}+\|c(x)-\nabla c(x)x\|\right)\cdot \diam(\mathcal{Z})}{(k+1)(k+2)}$$
after $k$ iterations. The cost of each iteration is dominated by a matrix vector multiplication, a projection onto $\mathcal{Z}$, and a proximal operation of $g$. 
Assuming that $\|\nabla c(x_k)\|$ and $\|c(x_k)-\nabla c(x_k)x_k\|$ are bounded throughout the algorithm, to obtain an $\varepsilon_k$ accurate solution to the subproblem \eqref{eqn:iterate} requires on the order of  $1/\sqrt{\varepsilon_k}$ such basic operations. Setting $\varepsilon_k\approx \frac{1}{k^{1+q}}$ for some $q>0$, we deduce $\sum^{\infty}_{i=1} \varepsilon_i\leq 1/q$. Thus we can find a point $x_{k}$ satisfying
$\|\bar{x}_{k+1}-x_k\|\approx \varepsilon$ after at most on the order of $\frac{1}{\varepsilon^{3+q}}\cdot\frac{1}{q^{(1+q)/2}}$ basic operations. 
One good choice is $q=\frac{1}{2\ln(1/\varepsilon)}$. 
To illustrate, for $\epsilon\approx 10^{-3}$, the complexity bound becomes on the order of $\frac{4}{\varepsilon^{3.07}}$. The meaning of $\|\bar{x}_{k+1}-x_k\|\approx \varepsilon$ in terms of ``near-stationarity'' becomes apparent with an application of Corollary~\ref{cor:quad_gr}.

\section{Inexact extensions \& model decrease as termination criteria}\label{sec:inexact}
Often, it may be impossible to obtain an exact minimizer $x^+$ of a model function $f_x$. 
What can one say then when $x^+$ minimizes the model function $f_{x}$ only approximately? By ``approximately'', one can mean a number of concepts. Two most natural candidates are that $x^+$ is {\em $\epsilon$-optimal}, meaning $f_{x}(x^+)\leq \inf f_x +\epsilon$, or that $x^+$ is {\em $\epsilon$-stationary}, meaning $|\nabla f_x(x^+)|\leq \epsilon$.
In both cases, all the results of Section~\ref{sec:main}  generalize quickly by bootstrapping Theorem~\ref{thm:main}; under a mild condition, both of the two notions above imply that $x^+$ is a minimizer of a slightly perturbed function, to which the key Theorem~\ref{thm:main} can be directly applied. 

\subsection{Near-optimality for the subproblems}
We begin with $\epsilon$-optimality, and discuss $\epsilon$-stationarity in Section~\ref{subsec:near-crit}. The following is an inexact analogue of Theorem~\ref{thm:main}. Though the statement may appear cumbersome at first glance, it simplifies dramatically in the most important case where $\omega$ is a quadratic; this case is recorded in Corollary~\ref{cor:quad_speci_func_err}.

\begin{theorem}[Perturbation result under approximate optimality]\label{thm:main_inexact}{\hfill \\ }
	Consider a closed function $f_x\colon\X\to\overline\R$ such that the inequality
	$$|f_x(y)-f(y)|\leq \omega(d(x,y)) \quad \textrm{ holds for all }  y\in\mathcal{X},$$
	where $\omega$ is some growth function.
Let $x^+$ be a point satisfying $f_x(x^+)\leq \inf f_x+\epsilon$. Then for any constant $\rho>0$, there exist two points $z$ and $\hat x$ satisfying the following.
\begin{enumerate}
	\item {\bf (point proximity)} The inequalities  $$\quad d(x^+,z)\leq  \frac{\epsilon}{\rho}\qquad \textrm{and} \qquad d(z,\hat x)\leq 2\cdot\frac{\omega(d(z,x))}{\omega'(d(z,x))} \qquad \textrm{hold},$$
	  under the convention $\frac{0}{0}=0$,
	\item {\bf (value proximity)} $\quad f(\hat x)\leq f(x^+)+2\omega(d(z,x))+\omega(d(x^+,x))$,
	\item {\bf (near-stationarity)} $\quad |\nabla f|(\hat x)\leq \rho+\omega'(d( z,x))+\omega'(d(\hat x,x))$.
\end{enumerate}
\end{theorem}
\begin{proof}
By Theorem~\ref{thm:eke}, for any $\rho>0$ there exists a point $z$ satisfying $f_x(z)\leq f_x(x^+)$, $d(z,x^+)\leq \frac{\epsilon}{\rho}$, and so that $z$ is the unique minimizer of the function $y\mapsto f_x(y)+\rho\cdot d(y,z)$. Define the functions 
$\widetilde f(y):=f(y)+\rho\cdot d(y,z)$ and $\widetilde f_x(y):=f_x(y)+\rho\cdot d(y,z)$. Notice the inequality 
$$|\widetilde{f}_x(y)-\widetilde{f}(y)|\leq \omega(d(x,y)) \quad \textrm{ for all }y.$$
Thus applying Theorem~\ref{thm:main}, we deduce that there exists a point $\hat x$ satisfying $d(z,\hat x)\leq 2\cdot\frac{\omega(d(z,x))}{\omega'(d(z,x))}$, $\widetilde{f}(\hat x)\leq \widetilde f(z)+\omega(d(z,x))$, and 
$|\nabla \widetilde f|(\hat x)\leq \omega'(d( z,x))+\omega'(d(\hat x,x))$. The point proximity claim is immediate. The value proximity follows from the inequality
\begin{align*}
f(\hat x)\leq  \widetilde{f}(\hat x)&\leq f(z)+\omega(d(z,x))\leq f_x(z)+2\omega(d(z,x))\leq f_x(x^+)+2\omega(d(z,x))\\
&\leq f(x^+)+2\omega(d(z,x))+\omega(d(x^+,x)).
\end{align*}
Finally, the inequalities
$$|\nabla f|(\hat x) \leq \rho+|\nabla \widetilde f|(\hat x) \leq \rho+ \omega'(d( z,x))+\omega'(d(\hat x,x))$$
imply the near-stationarity claim.
\end{proof}

Specializing to when $\omega$ is a quadratic yields the following.

\begin{corollary}[Perturbation  under quadratic error and approximate optimality]\label{cor:quad_speci_func_err}
	Consider a closed function $f_x\colon\X\to\overline\R$ and suppose that with some real $\eta>0$ the inequality
	$$|f_x(y)-f(y)|\leq \frac{\eta}{2}d^2(x,y) \quad \textrm{ holds for all }  y\in\mathcal{X}.$$
	Let $x^+$ be a point satisfying $f_x(x^+)\leq \inf f_x+\epsilon$. Then there exists a point $\hat x$ satisfying the following.
	\begin{enumerate}
		\item {\bf (point proximity)} $\quad d(x^+,\hat x)\leq \sqrt{\frac{4\epsilon}{3\eta}}+d(x^+,x),$
		\item {\bf (value proximity)} $\quad f(\hat x)\leq f(x^+)+\eta\left(\sqrt{\frac{\epsilon}{3\eta}}+d(x^+,x)\right)^2 +\frac{\eta}{2} d^2(x^+,x)$,
		\item {\bf (near-stationarity)} $\quad |\nabla f|(\hat x)\leq \sqrt{12\eta\epsilon}+3\eta\cdot d(x^+,x)$.
	\end{enumerate}
\end{corollary}
\begin{proof}
Consider the two point $\hat x$ and $z$ guaranteed to exist by Theorem~\ref{thm:main_inexact}. Observe the inequalities $$d(z,x)\leq d(z,x^+)+d(x^+,x)\leq \frac{\epsilon}{\rho}+d(x^+,x),$$
	and
	$$d(x^+,\hat x)\leq d(x^+,z)+d(z,\hat x)\leq \frac{\epsilon}{\rho}+d(z,x)\leq 2 \frac{\epsilon}{\rho}+d(x^+,x).$$
Hence we obtain $$f(\hat x)\leq  f(x^+)+\eta\left(\frac{\epsilon}{\rho}+d(x^+,x)\right)^2+\frac{\eta}{2}d^2(x^+,x),$$
and
$$|\nabla f|(\hat x)\leq \rho+\eta\left(\frac{\epsilon}{\rho}+d(x^+,x)\right)+\eta\cdot d(\hat x,x)\leq \left(\rho+\frac{3\eta \epsilon}{\rho}\right)+3\eta\cdot d(x^+,x).$$
Minimizing the right-hand-side of the last inequality in $\rho>0$ yields the choice $\rho=\sqrt{3\eta\epsilon}$. The result follows.
\end{proof}

%

An immediate consequence of Theorem~\ref{thm:main_inexact} is a subsequence converge result analogous to Corollary~\ref{cor:stat}.
\begin{corollary}[Subsequence convergence under  approximate optimality]\label{cor:subseq_inexact}
	Consider a sequence of points $x_k$ and closed functions $f_{x_k}\colon\mathcal X\to\overline{\R}$ satisfying  $d(x_{k+1},x_{k})\to 0$ and $f(x_{k+1})\leq \inf f_{x_k}+\epsilon_k$ for some sequence $\epsilon_k\to 0$. Suppose moreover that the inequality 
	$$|f_{x_k}(y)-f(y)|\leq \omega(d(y,x_k))\qquad \textrm{ holds for all indices } k \textrm{ and points } y\in \mathcal X,$$ 
	where $\omega$ is a proper growth function.
	If $(x^*,f(x^*))$ is a limit point of the sequence $(x_{k},f(x_k))$, then  $x^*$ is 
	stationary for $f$.
\end{corollary}
\begin{proof}
The proof is virtually identical to the proof of Corollary~\ref{cor:stat}, except that Theorem~\ref{thm:main_inexact} replaces Theorem~\ref{thm:main} with $\rho_k=\sqrt{\epsilon_k}$. We leave the details to the reader.
\end{proof}

\subsection{Model improvement as a stopping criterion}

The underlying premise of our work so far is that the step-size $d(x_{k+1},x_{k})$ can be reliably used to terminate the model-based algorithm in the sense of Theorem~\ref{thm:main}. We now prove that the same can be said for termination criteria based on the model decrease $\Delta_x:=f(x_k)-\inf f_{x_k}$. Indeed, this follows quickly by setting $x^+:=x$, $\epsilon:=\sqrt{\Delta_x}$, and $\rho$ a multiple of $\sqrt{\Delta_x}$ in Theorem~\ref{thm:main_inexact}.

\begin{corollary}[Perturbation result for model improvement]\label{cor:inex_eval_base}
	Consider a closed function $f_x\colon\X\to\overline\R$ such that the inequality
	$$|f_x(y)-f(y)|\leq \omega(d(x,y)) \quad \textrm{ holds for all }  y\in\mathcal{X},$$
	where $\omega$ is some growth function. Define the model improvement
	$\Delta_x:=f(x)-\inf f_{x}$. Then for any constant $c>0$, there exist two points $z$ and $\hat x$ satisfying the following.
	\begin{enumerate}
		\item {\bf (point proximity)} The inequalities $$\quad d(x,z)\leq  c^{-1}\sqrt{\Delta_x}\qquad \textrm{and} \qquad d(z,\hat x)\leq 2\cdot\frac{\omega(d(z,x))}{\omega'(d(z,x))}\qquad\textrm{hold},$$
		under the convention $\frac{0}{0}=0$,
	\item {\bf (value proximity)} $\quad f(\hat x)\leq f(x)+2\omega(d(z,x))$,
	\item {\bf (near-stationarity)} $\quad |\nabla f|(\hat x)\leq c\sqrt{\Delta_x}+\omega'(d( z,x))+\omega'(d(\hat x,x))$.
	\end{enumerate}
\end{corollary}
\begin{proof}
	Simply set	$x^+:=x$, $\epsilon:=\sqrt{\Delta_x}$, and $\rho=c\sqrt{\Delta_x}$ in Theorem~\ref{thm:main_inexact}.
\end{proof}

To better internalize the estimates, let us look at the case when $\omega$ is a quadratic.

\begin{corollary}[Perturbation for model improvement with quadratic error]\label{cor:perturb:func}
Consider a closed function $f_x\colon\X\to\overline\R$ and suppose that with some real $\eta>0$ the inequality
	$$|f_x(y)-f(y)|\leq \frac{\eta}{2}d^2(x,y) \quad \textrm{ holds for all }  y\in\mathcal{X}.$$
	Define the model decrease
	$$\Delta_x:=f(x)-\inf_y f_x(y).$$
Then there exists a point $\hat x$ satisfying 
\begin{enumerate}
	\item {\bf (point proximity)} $\quad d(\hat x,x)\leq  \sqrt{\frac{4}{3\eta}}\cdot\sqrt{\Delta_x}$,
	\item {\bf (value proximity)} $\quad f(\hat x)\leq f(x)+\frac{1}{3}\cdot\Delta_x$,
	\item {\bf (near-stationarity)} $\quad |\nabla f|(\hat x)\leq \sqrt{12\eta} \cdot \sqrt{\Delta_x}$.
\end{enumerate}
\end{corollary}
\begin{proof}
 Simply set $x^+:=x$ and $\epsilon:=\sqrt{\Delta_x}$ in Corollary~\ref{cor:quad_speci_func_err}.
\end{proof}

The subsequential convergence result in Corollary~\ref{cor:subseq_inexact} assumes that the step-sizes $d(x_{k+1},x_k)$ tend to zero. Now, it is easy to see that an analogous conclusion holds if instead the model improvements $f(x_k)-f_{x_k}(x_{k+1})$ tend to zero.

\begin{corollary}[Subsequence convergence under  approximate optimality II]
	Consider a sequence of points $x_k$ and closed functions $f_{x_k}\colon\mathcal X\to\overline{\R}$ satisfying $f_{x_k}(x_{k+1})\leq \inf f_{x_k}+\epsilon_k$ for some sequence $\epsilon_k\to 0$. Suppose that the inequality 
	$$|f_{x_k}(y)-f(y)|\leq \omega(d(y,x_k))\qquad \textrm{ holds for all indices } k \textrm{ and points } y\in \mathcal X,$$ 
	where $\omega$ is a proper growth function.
	Suppose moreover that the model improvements
	$f(x_k)-f_{x_k}(x_{k+1})$ tend to zero.
	If $(x^*,f(x^*))$ is a limit point of the sequence $(x_{k},f(x_k))$, then  $x^*$ is 
	stationary for $f$.
\end{corollary}

Following the pattern of the previous sections, we next pass to error-bounds. The following result shows that the slope error-bound implies that, not only do the step-sizes $d(x_{k+1},x_k)$
linearly bound the distance of $x_k$ to the stationary-point set (Theorem~\ref{thm:slop_subreg}), but so do the values $\sqrt{f(x_{k})-\inf f_{x_{k}}}$.

\begin{corollary}[Slope and model error-bounds]\label{cor:func_error}
Let $S$ be an arbitrary set and fix a point $x^*\in S$ satisfying the condition:
 \begin{itemize}
\item {\bf(Slope error-bound)}$\quad \dist(x;S)\leq L\cdot |\nabla f|(x)\quad \text{ for all }\quad x\in {\bf B}_{\gamma}(x^*)$.
\end{itemize}
Consider a closed function $f_{x}\colon\mathcal X\to\overline{\R}$ and suppose that for some $\eta>0$ the  inequality 
$$|f_{x}(y)-f(y)|\leq \frac{\eta}{2}d^2(y,x)\qquad \textrm{ holds for all }  y\in \mathcal X.$$ 
Then the following holds:
\begin{itemize}
\item {\bf(Model error-bound)}$$\dist(x;S)\leq \left(L\sqrt{12\eta}+\frac{2}{\sqrt{3\eta}}\right)\cdot \sqrt{f(x)-\inf f_x},$$ whenever $f(x)-\inf f_x< \frac{3\eta\gamma^2}{16}$ and $x$ lies in ${\bf B}_{\gamma/2}(x^*)$.
\end{itemize}
\end{corollary}
\begin{proof}
Suppose the inequality $f(x)-\inf f_x< \frac{3\eta\gamma^2}{16}$ holds and $x$ lies in ${\bf B}_{\epsilon/2}(x^*)$. Define $\Delta_x:=f(x)-\inf f_x$ and let $\hat x$ be the point guaranteed to exist by Corollary~\ref{cor:perturb:func}. We deduce 
\begin{align*}
d(\hat x,x^*)&\leq d(\hat x,x)+d(x,x^*)\leq \sqrt{\frac{4}{3\eta}}\cdot\sqrt{\Delta_x}+d(x,x^*)<\gamma.
\end{align*}
	Thus $\hat x$ lies in ${\bf B}_{\gamma}(x^*)$ and we obtain 
	\begin{align*}
	L\cdot |\nabla f|(\hat x)\geq \dist\left(\hat x; S\right)\geq \dist\left(x; S\right)-d(x,\hat x)\geq \dist(x;S)-\sqrt{\frac{4}{3\eta}}\cdot\sqrt{\Delta_x}.
	\end{align*}
	Taking into account the inequality $|\nabla f|(\hat x)\leq \sqrt{12\eta} \cdot \sqrt{\Delta_x}$, the result follows.
\end{proof}

Finally in the inexact regime, the slope error-bound (as in Theorem~\ref{thm:slop_subreg}) implies an {\em inexact} error-bound condition.

\begin{corollary}
	[Error-bounds under  approximate optimality]\label{cor:slope_subreg_approx_opt}{\hfill \\ }
Let $S$ be an arbitrary set and fix a point $x^*\in S$ satisfying the condition
 \begin{itemize}
\item {\bf(Slope error-bound)}$\quad \dist(x;S)\leq L\cdot |\nabla f|(x)\quad \text{ for all }\quad x\in {\bf B}_{\gamma}(x^*)$.
\end{itemize}
Consider a closed function $f_{x}\colon\mathcal X\to\overline{\R}$ and suppose that for some $\eta>0$ the  inequality 
$$|f_{x}(y)-f(y)|\leq \frac{\eta}{2}d^2(y,x)\qquad \textrm{ holds for all }  y\in \mathcal X.$$ 
 Define the constant $\mu:=2\sqrt{L(5L\eta+4)}$.
Then letting $x^+$ be any point satisfying $f_x(x^+)\leq \inf f_x +\epsilon$, the following two error-bounds hold:
	\begin{itemize}
		\item {\bf(Step-size error-bound)}$$ \dist(x;S)\leq \mu\sqrt{\epsilon}+(7L\eta+6)\cdot d(x^+,x)$$
		whenever $\sqrt{\epsilon}< \gamma\mu/12L$, $d(x^+,x)<\gamma/9$, and
		$x^+$ lies in ${\bf B}_{\gamma/3}(x^*)$.
		\item {\bf(Model error-bound)}
		$$\dist(x;S)\leq \left(L\sqrt{12\eta}+\frac{2}{\sqrt{3\eta}}\right)\sqrt{f(x)-f_{x}(x_+)+\epsilon}.$$
		whenever $f(x)-\inf f_x< \frac{3\eta\gamma^2}{16}$ and $x$ lies in ${\bf B}_{\gamma/2}(x^*)$.
	\end{itemize}
\end{corollary}
\begin{proof}
	Consider two points $x,x^+$ satisfying 
	 $\sqrt{\epsilon}\leq \gamma\mu/12L$, $d(x^+,x)<\gamma/9$, and
		$x^+\in{\bf B}_{\gamma/3}(x^*)$.
	Let $\hat x, z$ be the points guaranteed to exist by Corollary~\ref{cor:quad_gr} for some $\rho$; we will decide on the value of $\rho>0$ momentarily. First, easy manipulations using the triangle inequality yield
	\begin{align*}
	d(\hat x,z)&\leq d(z,x),\qquad\qquad~~ d(z,x^+)\leq d(z,x)+d(x^+,x), \\ d(z,x)&\leq \epsilon/\rho+d(x^+,x), \qquad d(x^+,\hat x)\leq 4\epsilon/\rho+5d(x^+,x).
	\end{align*}
	Suppose for the moment $\hat x$ lies in ${\bf B}_{\gamma}(x^+)$; we will show after choosing $\rho$ appropriately that this is the case.
	Then we obtain the inequality
	\begin{align*}
	L\cdot |\nabla f|(\hat x)\geq \dist\left(\hat x; S\right)&\geq \dist\left(x; S\right)-d(x^+,\hat x)-d(x^+,x)\\
	&\geq \dist(x;S)- 4\epsilon/\rho_k-6d(x^+,x).
	\end{align*}	
	Taking into account the inequality 
	$$|\nabla f|(\hat x)\leq \rho+\eta(d(z,x)+d(\hat x,x))\leq \rho +\eta(5\epsilon/\rho+7d(x^+,x)),$$ we conclude
	$$\dist\left(x; S\right)\leq L\rho+\frac{5L\eta\epsilon}{\rho}+\frac{4\epsilon}{\rho}+(7L\eta+6)\cdot d(x^+,x),$$ as claimed. Minimizing the right-hand-side in $\rho$ yields the choice $\rho:=\sqrt{\frac{(5L\eta+4)\epsilon}{L}}$. With this choice, the inequality above becomes
	$$\dist\left(x; S\right)\leq 2\sqrt{L(5L\eta+4)\epsilon}+(7L\eta+6)\cdot d(x^+,x).$$
	Finally, let us verify that $\hat{x}$ indeed lies in ${\bf B}_{\gamma}(x^*)$.
	To see this, simply observe
		\begin{align*}
		d(\hat x,x^*)&\leq d(\hat x,z) +d(z,x^+)+d(x^+,x^*)\leq 2d(z,x)+d(x^+,x)+d(x^+,x^*)\\
		&\leq 2\epsilon_k/\rho_k+3d(x^+,x)+d(x^+,x^*)< \gamma.
		\end{align*}
		The result follows. The step-size error bound condition follows. The functional error-bound is immediate from Corollary~\ref{cor:func_error}.
\end{proof}

In particular, in the notation of Corollary~\ref{cor:slope_subreg_approx_opt}, if one wishes the error $d(x^+,x)$ to linearly bound the distance $d(x;S)$, then one should ensure that the tolerance $\epsilon$ is on the order of $d^2(x^+,x)$.

\subsection{Near-stationarity for the subproblems}\label{subsec:near-crit}
In this section, we explore the setting where $x^+$ is only $\epsilon$-stationary for $f_{x}$. To make progress in this regime, however, we must first assume a linear structure on the metric space. We suppose throughout that $\X$ is a Banach space, and denote its dual by $\mathcal X^*$. For any dual element $v\in \X^*$ and a point $x\in \X$, we use the notation $\langle v,x\rangle:=v(x)$. 
 Second, the property $|\nabla f_x(x^+)|\leq \epsilon$ alone appears to be too weak. Instead, we will require a type of uniformity in the slopes.
In the simplest case, we will assume that $x^+$ is such that the function $f_x$ majorizes the simple quadratic 
$$f_x(x^+)+\langle v^*,\cdot-x^+\rangle-\eta\|\cdot-x^+\|^2$$
where  $v\in \X^*$ is some dual element satisfying $\|v^*\|\leq \epsilon$. In the language of variational analysis, $v$ is a {\em proximal subgradient} of $f_x$ at $x^+$; see e.g. \cite{CLSW,prox_subgrad}. A quick computation immediately shows the inequality $|\nabla f_x|(x^+)\leq \epsilon$. Assuming that $\eta$ is uniform throughout the iterative process will allow us to generalize the results of Section~\ref{sec:main}. Such uniformity is immediately implied by prox-regularity \cite{prox_reg} for example -- a broad and common setting for nonsmooth optimization.

\begin{corollary}[Perturbation result under approximate stationarity]\label{cor:capr_perturb2}
	Consider a closed function $f_x\colon\X\to\overline\R$ on a Banach space $\X$ such that the inequality
	$$|f_x(y)-f(y)|\leq \omega_1(d(x,y)) \quad \textrm{ holds for all }  y\in\X,$$
	where $\omega_1$ is some growth function. Suppose moreover that for some point $x^+\in \X$, a dual element  $v\in \X^*$, and a growth function $\omega_2$, the inequality
	 	$$f_x(y)\geq f_x(x^+)+\langle v,y-x^+\rangle -\omega_2(d(y,x^+))\qquad \textrm{ holds for all }y\in \X.$$
 Then there exists a point $\hat x$ satisfying	
	\begin{enumerate}
		\item {\bf (point proximity)} $\quad d(x^+,\hat{x})\leq 2\cdot \frac{\omega_1(d(x^+,x))}{\omega'_1(d(x^+,x))}$,
		\item {\bf (value proximity)} $\quad f(\hat x)\leq f(x^+)+\langle v,\hat x-x^+\rangle+ \omega_1(d(x^+,x))-\omega_2(d(\hat x,x^+))$,
		\item {\bf (near-stationarity)} $\quad |\nabla f|(\hat x)\leq \|v\|+ \omega'_1(d(x^+,x))+\omega'_1(d(\hat x,x))+\omega'_2(d(\hat x,x^+))$.
	\end{enumerate}
\end{corollary}
\begin{proof}
Define the functions $\widetilde f(y):=f(y)-\langle v,y-x^+\rangle +\omega_2(d(y,x^+))$  	and $\widetilde f_x(y):=f_x(y)-\langle v,y-x^+\rangle +\omega_2(d(y,x^+))$. Note that $x^+$ minimizes $\widetilde f_x$ and that the inequality
	$$|\widetilde f_x(y)-\widetilde f(y)|\leq \omega_1(d(x,y)) \quad \textrm{ holds for all }  y\in \mathcal{X}.$$
	Applying Theorem~\ref{thm:main}, we obtain a point $\hat x$ satisfying the point proximity claim, along with the inequalities $\widetilde f(\hat x)\leq \widetilde f(x^+)+\omega_1(d(x^+,x))$ and 
$|\nabla \widetilde f|(\hat x)\leq \omega'_1(d(x^+,x))+\omega'_1(d(\hat x,x))$.
The value proximity claim follows directly from definitions, while the near-stationarity is immediate from the inequality,
$|\nabla \widetilde f|(\hat x)\geq |\nabla f|(\hat x)-\|v\|-\omega'_2(d(\hat x,x^+)).$
The result follows.
\end{proof}

As an immediate consequence, we obtain the subsequence convergence result.

\begin{corollary}[Subsequence convergence under  approximate optimality]\label{cor:subseq_inexact3}
	Consider a sequence of points $x_k$  and closed functions $f_{x_k}\colon\mathcal X\to\overline{\R}$ satisfying 
	$$|f_{x_k}(y)-f(y)|\leq \omega_1(d(y,x_k))\qquad \textrm{ for all indices } k \textrm{ and points } y\in \mathcal X,$$ 
	where $\omega_1$ is some proper growth function.
 Suppose that the inequality
	 	$$f_{x_{k}}(y)\geq f_{x_k}(x_{k+1})+\langle v_{k+1},y-x_{k+1}\rangle -\omega_2(d(y,x_{k+1}))\qquad\qquad \textrm{ holds for all }k \textrm{ and all }y\in \X,$$
	 	where $\omega_2$ is some proper growth function and $v_k\in\X^*$ are some dual elements. Assume moreover that $d(x_{k+1},x_k)$ and $\|v_k\|$ tend to zero.  
	If $(x^*,f(x^*))$ is a limit point of the sequence $(x_{k},f(x_k))$, then  $x^*$ is 
	stationary for $f$.
\end{corollary}

Finally, the following inexact error-bound result holds, akin to Theorem~\ref{thm:slop_subreg}.
\begin{corollary}[Error-bounds under approximate stationarity]
Let $S$ be an arbitrary set and fix a point $x^*\in S$ satisfying the condition
 \begin{itemize}
\item {\bf(Slope error-bound)}$\quad \dist(x,S)\leq L\cdot |\nabla f|(x)\quad \text{ for all }\quad x\in {\bf B}_{\gamma}(x^*)$.
\end{itemize}
Consider a closed function $f_{x}\colon\mathcal X\to\overline{\R}$ and suppose that for some $\eta>0$ the  inequality 
$$|f_{x}(y)-f(y)|\leq \frac{\eta}{2}\cdot\|y-x\|^2\qquad \textrm{ holds for all }  y\in \mathcal X.$$ 
	Fix a point $x^+$ and a dual element $v\in \X^*$ so that the inequality 
	$$f_{x}(y)\geq f_x(x^+)+\langle v,y-x^+\rangle -\frac{\eta}{2}\|y-x^+\|^2\qquad \textrm{ holds for all }y\in \X.$$
Then the approximate error-bound holds:
	\begin{itemize}
		\item {\bf(Step-size error-bound)}$\quad\dist\left(x; S\right)\leq L\|v\|+(4\eta L+2) \|x^+-x\|\quad \textrm{ when }x,x^+\in {\bf B}_{\gamma/3}(x^*).$
	\end{itemize}
\end{corollary}
\begin{proof}
The proof is entirely analogous to that of Theorem~\ref{thm:slop_subreg}. 
Consider two points $x,x^+\in {\bf B}_{\gamma/3}(x^*)$. Let $\hat x$ be the point guaranteed to exist by Corollary~\ref{cor:capr_perturb2}. We deduce 
\begin{align*}
d(\hat x,x^*)&\leq d(\hat x,x^+)+d(x^+,x^*)\leq d(x^+,x)+d(x^+,x^*)<\gamma.
\end{align*}
Thus $\hat x$ lies in ${\bf B}_{\epsilon}(x^*)$ and we deduce 
\begin{align*}
L\cdot |\nabla f|(\hat x)\geq \dist\left(\hat x; S\right)&\geq \dist\left(x; S\right)-d(x^+,\hat x)-d(x^+,x)\\
&\geq  \dist\left(x; S\right)-2d(x^+,x).
\end{align*}
Taking into account the inequality $|\nabla f|(\hat x)\leq \|v\|+4\eta \|x^+-x\|$, we conclude
$$\dist\left(x; S\right)\leq L\|v\|+(4\eta L+2)\cdot \|x^+-x\|,$$ as claimed.
\end{proof}

\subsection*{Conclusion}
In this paper, we considered a general class of nonsmooth minimization algorithms that use Taylor-like models. We showed that both the step-size and the improvement in the model's value can be used as reliable stopping criteria. We deduced subsequence convergence to stationary points, and error-bound conditions under natural regularity properties of the function. The results fully generalized to the regime where the models are minimized inexactly. 
Ekeland's variation principle (Theorem~\ref{thm:eke}) underlies all of our current work. Despite the wide uses of the principle in variational analysis, its impact on convergence of basic algorithms, such as those covered here and in \cite{error_bound_d_lewis,alt_cur}, is not as commonplace as it should be. We believe that this work takes an important step towards rectifying this disparity and the techniques presented here will pave the way for future algorithmic insight.

%
%

\nopagebreak
\bibliographystyle{plain}
\bibliography{dima_bib}

\end{document}